\documentclass{amsart}
\usepackage{amsfonts,amssymb,amsmath,amsthm}
\usepackage{url}
\usepackage{enumerate}

\makeatletter
\@namedef{subjclassname@2010}{  \textup{2010} Mathematics Subject Classification.}
\makeatother

\theoremstyle{definition}

\numberwithin{equation}{section}
\email{hzoubeir2014@gmail.com}
\input{tcilatex}

\begin{document}
\address{ }
\author{Hicham Zoubeir}
\address{Ibn Tofail University, Department of Mathematics,\\
Faculty of Sciences, P. O. B : $133,$ Kenitra, Morocco.}
\title[About a conjecture on difference equations ]{About a conjecture on
difference equations in quasianalytic Carleman classes}

\begin{abstract}
In this paper we consider the difference equation $(E):\overset{q}{\underset{%
j=1}{\sum }}a_{i}(x)\varphi (x+\alpha _{i})=\chi (x)$ where $\alpha
_{1}<...<\alpha _{q}$ \ $(q\geq 3)$ are given real constants, $a_{j}$ $%
(j=1,...,q)$ are given holomorphic functions on some strip $%
\mathbb{R}
_{\delta }$ $(\delta >0)$ such that $a_{1}$ and $a_{q}$ are nowhere
vanishing on $%
\mathbb{R}
_{\delta },$ and $\chi $ a function which belongs to a quasianalytic
Carleman class $C_{M}\{%
\mathbb{R}
\}.$ We prove under a growth condition on the functions $a_{j}$ that the
equation $(E)$ is solvable in the class $C_{M}\{%
\mathbb{R}
\}.$
\end{abstract}

\dedicatory{$\emph{This}$ $\emph{modest}$ $\emph{work}$ $\emph{is}$ $\emph{%
dedicated}$ $\emph{to}$ $\emph{the}$ $\emph{memory}$ $\emph{of}$ $\emph{our}$
$\emph{beloved}$ $\emph{master}$ $\emph{Ahmed}$ $\emph{Intissar}$ $\emph{%
(1951-2017),}$ $\emph{a}$ $\emph{distinguished}$ $\emph{professor,a}$ $\emph{%
brilliant}$ $\emph{mathematician,a}$ $\emph{man}$ $\emph{with}$ $\emph{a}$ $%
\emph{golden}$ $\emph{heart.}$\emph{\ }}
\maketitle

\section{Introduction}

In the paper (\cite{BEL}), G. Belitskii, E. M. Dyn'kin and V. Tkachenko have
formulated the following conjecture :

\textbf{"}\textit{Let }$\chi ,\ a_{i},\ i=1,...,q,$\textit{\ be functions in
a Carleman class }$C_{M}(%
\mathbb{R}
)$\textit{\ such that }$a_{1}$\textit{and }$a_{q}$\textit{\ are nowhere
vanishing on }$%
\mathbb{R}
,$\textit{\ and }$\alpha _{1}<...<\alpha _{q}$\textit{\ some real numbers.
Then the difference equation }$:$%
\begin{equation*}
(E):\overset{q}{\text{ }\underset{j=1}{\sum }}a_{i}(x)\varphi (x+\alpha
_{i})=\chi (x)
\end{equation*}%
\textit{is solvable in the Carleman class }$C_{M}\{%
\mathbb{R}
\}."$

In that paper the authors, relying on a result of decomposition in Carleman
classes, have proved the conjecture in the particular cases where the
coefficients $a_{j}$ are constants or when the coefficients are variables
with $q=2.$ They have also suggested that the same method could be used to
show the solvability of the equation $(E)$ in a quasianalytic Carleman class 
$C_{M}\{%
\mathbb{R}
\},$ if we assume that the functions $\frac{1}{a_{1}},$ $\frac{1}{a_{q}},$ $%
\frac{a_{2}}{a_{1}},...,\frac{a_{q}}{a_{1}},$ $\frac{a_{1}}{a_{q}},...,$ $%
\frac{a_{q-1}}{a_{q}}$ $(q\geq 3)$ can be continued in some strip $%
\mathbb{R}
_{\delta }:=\{z\in 
\mathbb{C}
:|\func{Im}(z)|<\delta \}$ as analytic functions increasing on $%
\mathbb{R}
_{\delta }$ not too fast at infinity. As an example of such coefficients,
they have mentioned the class of rational functions. In this paper, our
purpose is to give a precise meaning to this assertion, by proving that the
result is true even if the functions $\frac{1}{a_{1}},$ $\frac{1}{a_{q}},$ $%
\frac{a_{2}}{a_{1}},...,\frac{a_{q}}{a_{1}},\frac{a_{1}}{a_{q}},...,\frac{%
a_{q-1}}{a_{q}}$ have more rapide increase at infinity provided that it is
of the form $e^{e^{C|\func{Re}(z\text{ })|}}$ where $C>0$ is a constant.

\section{\textbf{Notations, definitions and statement of the main result}}

We set for every $\rho >0,$ $a\geq 0:$ 
\begin{equation*}
\left \{ 
\begin{array}{c}
\mathbb{R}
_{\rho }:=\{z\in 
\mathbb{C}
:|\func{Im}(z)|<\rho \},\text{ }%
\mathbb{R}
_{\rho }^{\pm }:=\{z\in 
\mathbb{R}
_{\rho }:\pm \func{Re}(z)>\rho \} \\ 
\mathbb{R}
_{\rho ,a}:=\{z\in 
\mathbb{R}
_{\rho }:|\func{Re}(z)|\leq a\} \\ 
\Delta _{\rho }:=\{z\in 
\mathbb{C}
:|z|<\rho \},\text{ }\Delta _{\rho }^{\pm }:=\{z\in \Delta _{\rho }:\pm 
\func{Re}(z)\leq 0\} \\ 
\Gamma _{\rho }:=\{z\in 
\mathbb{C}
:|z|=\rho \},\  \Gamma _{\rho }^{\pm }:=\{z\in \Gamma _{\rho }:\pm \func{Re}%
(z)\leq 0\}%
\end{array}%
\right. \text{ }
\end{equation*}

For every non empty subset $V$ of $%
\mathbb{C}
$ and every $z\in 
\mathbb{C}
$ and $n\in 
\mathbb{N}
^{\ast }$we set : 
\begin{equation*}
\left \{ 
\begin{array}{c}
V^{\text{ }(0)}:=V,\text{ }V^{\text{ }(n)}:=\{u_{1}+...+u_{n}\text{ }%
:u_{j}\in V,\text{ }j=1,...,n\text{ }\},\text{ }n\geq 1 \\ 
z+V:=\{z+u:u\in V\text{ }\},\text{ }z-V:=\{z-u:u\in V\text{ }\}%
\end{array}%
\right.
\end{equation*}

$dm(\zeta )$ denotes the Lebesgue mesure on $\mathbb{C}.$

Let $S$\ be a nonempty subset of $%
\mathbb{C}
.$\ By $O(S$\ $)$\ we denote the set of holomorphic functions on some
neighborhood of $S$.

Let $F:U\subset 
\mathbb{C}
\rightarrow 
\mathbb{C}
$ be a function of class $C^{1}$ on an open subset $U$ of $%
\mathbb{C}
.$ We set for all $z\in U:$%
\begin{equation*}
\overline{\partial }F(z):=\frac{1}{2}[\frac{\partial F}{\partial x}(z)+i%
\text{ }\frac{\partial F}{\partial y}(z)]
\end{equation*}%
$\overline{\partial }$ is called the operator of Cauchy-Riemann.

Let $M:=(M_{n})_{n\in 
\mathbb{N}
}$ be a sequence of strictly positive numbers. Let $M:=(M_{n})_{n\text{ }%
\geq 0\text{ }}$be a sequence of strictly positive real numbers. The
Carleman class $C_{M}$ $\{%
\mathbb{R}
\}$ is the set of all functions $f:W\rightarrow 
\mathbb{C}
$ of class $C^{\infty }$such that%
\begin{equation*}
||f^{\text{ }(n)}||_{\infty ,I}\leq C_{I}\rho _{I}^{n}M_{n},\text{ }n\in 
\mathbb{N}%
\end{equation*}%
for every compact interval $I$ $\ $contained in $%
\mathbb{R}
$ with some constants $C_{I},\rho _{I}>0.$

The Carleman class $C_{M}$ $\{%
\mathbb{R}
\}$ is said to be quasinalytic if every function $f\in C_{M}$ $\{%
\mathbb{R}
\}$ such that $f^{\text{ }(n)}(u)=0$ for some $u$ $\in 
\mathbb{R}
$ $\ $and every $n\in 
\mathbb{N}
$ is identically equal to $0.$

The Carleman class $C_{M}$ $\{%
\mathbb{R}
\}$ is called regular if the following conditions hold :%
\begin{equation*}
\left \{ 
\begin{array}{c}
\left( \frac{M_{n+1}}{(n+1)!}\right) ^{2}\leq \frac{M_{n}}{n!}\frac{M_{n+2}}{%
(n+2)!},\text{ }n\in 
\mathbb{N}
\\ 
\underset{n\in 
\mathbb{N}
}{\sup }\left( \frac{M_{n+1}}{(n+1)M_{n}}\right) ^{\frac{1}{n}}<+\infty \\ 
\underset{n\rightarrow +\infty }{\lim }M_{n}^{\frac{1}{n}}=+\infty%
\end{array}%
\right.
\end{equation*}

To the Carleman $C_{M}$ $\{%
\mathbb{R}
\}$ we associate its weight $H_{M}$ defined by the following relation :%
\begin{equation*}
H_{M}(x):=\underset{n\in 
\mathbb{N}
}{\inf }\left( \frac{M_{n}}{n!}x^{n}\right) ,\text{ }x>0
\end{equation*}

In this paper the following result will play a crucial role.

\begin{theorem}
$($\cite{DYN}$)$ \textit{We assume that the Carleman class }$C_{M}\{%
\mathbb{R}
\}$\textit{\ is regular}. \textit{A function }$f:%
\mathbb{R}
\rightarrow 
\mathbb{C}
$\textit{\ belongs to }$C_{M}\{%
\mathbb{R}
\}$\textit{\ if and only if there exists for every compact interval }$I$%
\textit{\ of }$%
\mathbb{R}
$\textit{\ a compactly supported function }$F_{I}:%
\mathbb{C}
\rightarrow 
\mathbb{C}
$\textit{\ of class }$C^{1}$\textit{\ such that }$F_{I}$\textit{\ is an
extension of }$f$\textit{\ and satisfies the following estimate }$:$\textit{%
\ }%
\begin{equation*}
|\overline{\partial }F_{I}(z)|\leq A_{I}H_{M}(B_{I}|\func{Im}(z)|),\text{ }%
z\in 
\mathbb{C}
\end{equation*}%
\textit{where }$A_{I},$\textit{\ }$B_{I}>0$\textit{\ are constants.}
\end{theorem}

Throughout this paper, we assume that the Carleman class $C_{M}\{%
\mathbb{R}
\}$ is regular and quasianalytic.

Our main result in this paper is the following.

\begin{theorem}
\textit{Let }$q\in 
\mathbb{N}
^{\ast }\backslash \{1,2\},\  \delta >0,$ $\chi \in C_{M}\{%
\mathbb{R}
\}$\textit{\ and }$a_{j}$\textit{\ }$\in O(%
\mathbb{R}
_{\delta })$\textit{\ }$(j=1,...,q)$\textit{\ such that }$a_{1}$ \textit{and 
}$a_{q}$\textit{\ are nowhere vanishing on }$%
\mathbb{R}
_{\delta }.$\textit{\ We assume that the following growth condition holds }$:
$%
\begin{equation}
\underset{z\in 
\mathbb{R}
_{\delta }}{\sup }\left( \left( \sum_{j=2}^{q}\left \vert \frac{a_{j}(z)}{%
a_{1}(z)}\right \vert +\sum_{j=1}^{q-1}\left \vert \frac{a_{j}(z)}{a_{q}(z)}%
\right \vert +\frac{1}{|a_{1}(z)|}+\frac{1}{|a_{q}(z)|}\right) e^{-e^{C|\func{%
Re}z|}}\right) <+\infty   \label{hyp}
\end{equation}%
for some constant $C>0.$ \textit{Then the difference equation }$:$%
\begin{equation*}
(E):\overset{q}{\text{ }\underset{j=1}{\sum }}a_{i}(x)\varphi (x+\alpha
_{i})=\chi (x)
\end{equation*}%
\textit{is solvable in the Carleman class }$C_{M}\{%
\mathbb{R}
\}.$
\end{theorem}

\section{Proof of the main result}

Let us first prove the following lemma.

\begin{lemma}
\textit{Given }$f\in C_{M}\{%
\mathbb{R}
\},$\textit{\ }$C_{0}>0$ \textit{and }$\rho \in $\textit{\ }$\left] 0,\frac{%
\pi }{2C_{0}}\right[ ,$ \textit{there exist two functions}$f_{\pm }:$\textit{%
\ }$(\mathbb{C}\backslash \Delta _{\rho }^{\pm })\cup 
\mathbb{R}
\longrightarrow \mathbb{C}$ \textit{which are} \textit{holomorphic on }$%
\mathbb{C}\backslash (\Gamma _{\rho }^{\pm }\cup \Delta _{\rho }^{\pm }),$ 
\textit{whose restrictions to} $%
\mathbb{R}
$ \textit{belong to }$C_{M}\{%
\mathbb{R}
\},$\textit{and such that the following conditions hold }$:$%
\begin{eqnarray*}
f(x) &=&f_{+}(x)+f_{-}(x),\text{ }x\in \lbrack -\rho ,\rho ] \\
|f_{+}(z)| &\leq &D_{0}e^{-\cos (\rho C_{0})e^{C|\func{Re}(z\text{ })|}},%
\text{ }z\in 
\mathbb{R}
_{\rho }^{+} \\
|f_{-}(z)| &\leq &D_{0}e^{-\cos (\rho C_{0})e^{C|\func{Re}(z\text{ })|}},%
\text{ }z\in 
\mathbb{R}
_{\rho }^{-}
\end{eqnarray*}%
\textit{where }$D_{0}>0$\textit{\ is a constant.}
\end{lemma}

\begin{proof}
Since $f$ belongs to $C_{M}\{%
\mathbb{R}
\}$ there exists, according to Dyn'kin's theorem (\cite{DYN}), a compactly
supported function $F:%
\mathbb{C}
\rightarrow 
\mathbb{C}
$ of class $C^{1}$ such that $F$ is an extension of the restriction of $f$
to the interval $[-\rho ,\rho ]$ and satisfies the following estimate :%
\begin{equation*}
|\overline{\partial }F(z)|\leq AH_{M}(B|\func{Im}(z)|),\text{ }z\in 
\mathbb{C}%
\end{equation*}%
where $A,$ $B>0$ are constants. Following the same approach as that of (\cite%
{BEL}, pages $34,35),$ but using the Cauchy-Pompeiu formula on the disk $%
\Delta _{\rho },$ for the function $e^{e^{C\text{ }_{0}z}+e^{-C_{0}\text{ }%
z}}f(z),$ we show that the functions :%
\begin{equation*}
\left \{ 
\begin{array}{c}
\begin{array}{c}
f_{+}(z)=\frac{1}{2i\pi }e^{-e^{C_{0}\text{ }z}-e^{-C_{0}z}}\int_{\Gamma
_{\rho }^{+}}\frac{e^{e^{C_{0}\zeta }+e^{-C_{0}\zeta }}F(\zeta )}{\zeta -z}%
d\zeta - \\ 
-\frac{1}{\pi }e^{-e^{C\text{ }z}-e^{-C\text{ }z}}\iint_{\Delta _{\rho }^{+}}%
\frac{e^{e^{C\text{ }_{0}\zeta }+e^{-C_{0}\zeta }}\overline{\partial }%
F(\zeta )}{\zeta -z}dm(\zeta )%
\end{array}
\\ 
f_{-}(z)=\frac{1}{2i\pi }e^{-e^{C_{0}\text{ }z}-e^{-C_{0}\text{ }%
z}}\int_{\Gamma _{\rho }^{-}}\frac{e^{e^{C_{0}\zeta }+e^{-C_{0}\zeta
}}F(\zeta )}{\zeta -z}d\zeta - \\ 
-\frac{1}{\pi }e^{-e^{C_{0}\text{ }z}-e^{-C_{0}\text{ }z}}\iint_{\Delta
_{\rho }^{-}}\frac{e^{e^{C_{0}\text{ }\zeta }+e^{-C_{0}\zeta }}\overline{%
\partial }F(\zeta )}{\zeta -z}dm(\zeta )%
\end{array}%
\right.
\end{equation*}%
satisfy the required conditions.\ 
\end{proof}

Now we set :%
\begin{equation*}
\left \{ 
\begin{array}{c}
\beta _{j}:=\alpha _{j}-\alpha _{1},\text{ }j=2,...,q \\ 
b_{j}(z):=-\frac{a_{j}(z)}{a_{1}(z)},\text{ }z\in 
\mathbb{R}
_{\delta },\text{ }j=2,...,q \\ 
\gamma _{j}:=\alpha _{q}-\alpha _{j},\text{ }j=1,...,q-1 \\ 
c_{j}(z):=-\frac{a_{j}(z)}{a_{q}(z)},\text{ }z\in 
\mathbb{R}
_{\delta },\text{ }j=1,...,q-1%
\end{array}%
\right.
\end{equation*}%
Let $C_{1}>C$ \ and $\delta _{0}\in \left] 0,\min \left( \delta ,\frac{\pi }{%
2C_{1}}\right) \right[ .$ Then according to the lemma above, there exist
onstant $D_{1}>0$ and two functions $\chi _{\pm }:(%
\mathbb{C}
\backslash \Delta _{\delta _{0}}^{\pm })\cup 
\mathbb{R}
\longrightarrow 
\mathbb{C}
$ \ which are holomorphic on $%
\mathbb{C}
\backslash (\Gamma _{\delta _{0}}^{\pm }\cup \Delta _{\delta _{0}}^{\pm }),$
whose restrictions to $%
\mathbb{R}
$ belong to $C_{M}\{%
\mathbb{R}
\},$ and such that the following conditions hold :%
\begin{equation}
\left \{ 
\begin{array}{c}
\chi (x)=\chi _{+}(x)+\chi _{-}(x),\text{ }x\in \lbrack -\delta _{0},\delta
_{0}] \\ 
|\chi _{+}(z)|\leq D_{1}e^{-\cos (C_{1}\delta _{0})e^{C_{1}|\func{Re}(z\text{
})|}},\text{ }z\in 
\mathbb{R}
_{\delta _{0}}^{+} \\ 
|\chi _{-}(z)|\leq D_{1}e^{-\cos (C_{1}\delta _{0})e^{C_{1}|\func{Re}(z\text{
})|}},\text{ }z\in 
\mathbb{R}
_{\delta _{0}}^{-}\text{ }%
\end{array}%
\right.  \label{SPLIT}
\end{equation}%
Let $(g_{n})_{n\in 
\mathbb{N}
\text{ }}$and $(h_{n})_{n\in 
\mathbb{N}
\text{ }}$ be the sequences of complex valued functions defined on the strip 
$%
\mathbb{R}
_{\delta _{0}}$ by the formulas :%
\begin{equation}
\left \{ 
\begin{array}{c}
g_{_{0}}(z):=\frac{\chi _{+}(z)}{a_{1}(z)},\text{ }g_{n+1}(z):=\underset{j=2}%
{\overset{q}{\sum }}b_{j}(z)g_{n}(z+\beta _{j}) \\ 
h_{_{0}}(z):=\frac{\chi _{-}(z)}{a_{q}(z)},\text{ }h_{n+1}(z):=\overset{q-1}{%
\underset{j=1}{\sum }}c_{j}(z)h_{n}(z-\gamma _{j})%
\end{array}%
\right.  \notag
\end{equation}%
It is clear that all the functions $g_{n|%
\mathbb{R}
}$ and $h_{n|%
\mathbb{R}
}$ belong to $C_{M}$ $(%
\mathbb{R}
).$ Let us set :\  \  \  \  \  \  \  \  \  \  \  \  \  \  \  \  \  \  \  \  \  \  \  \  \  \  \  \  \  \
\  \  \  \  \  \  \  \  \  \  \  \  \  \  \  \  \  \  \  \  \  \  \  \  \  \  \  \  \  \  \  \  \  \  \  \  \ 
\begin{equation*}
\left \{ 
\begin{array}{c}
K_{1}:=\{ \beta _{j}:j\in 2,...,q\} \\ 
K_{2}:=\{ \gamma _{j}:j\in 1,...,q-1\}%
\end{array}%
\right.
\end{equation*}%
It follows from (\ref{hyp}) that we have for evey $n\in 
\mathbb{N}
$ $,$ $z\in 
\mathbb{R}
_{\delta _{0}}:$

\begin{equation*}
\left \{ 
\begin{array}{c}
|g_{n+1}(z)|\text{ }\leq e^{Le^{C|\func{Re}(z\text{ })|}}\underset{u\text{ }%
\in \text{ }z\text{ }+K}{\max }|g_{n}(u)| \\ 
|h_{n+1}(z)|\text{ }\leq e^{Le^{C|\func{Re}(z\text{ })|}}\underset{u\text{ }%
\in \text{ }z\text{ }-K}{\max }|h_{n}(u)|%
\end{array}%
\right.
\end{equation*}%
where $L>1$\ is a constant. Thence we have for all $n\in 
\mathbb{N}
^{\ast },$\ $z\in 
\mathbb{R}
_{\delta _{0}}:$%
\begin{equation*}
\left \{ 
\begin{array}{c}
\begin{array}{c}
g_{n}(z)|\text{ }\leq e^{\underset{j=0}{\overset{n-1}{\sum }}Le^{C\text{ }(|%
\func{Re}(z\text{ })|+j\beta _{q})}}\underset{u\text{ }\in \text{ }z\text{ }%
\mathbf{+}K_{1}^{(n)}}{\max }|g_{\text{ }0}(u)| \\ 
\leq \text{ }e^{nLe^{C\text{ }(|\func{Re}(z\text{ })|+n\text{ }\beta _{q})}}%
\underset{u\text{ }\in \text{ }z\text{ }\mathbf{+}K_{1}^{(n)}}{\max }|\chi _{%
\text{ }+}(u)|%
\end{array}
\\ 
|h_{n}(z)|\text{ }\leq e^{\underset{j=0}{\overset{n-1}{\sum L}}e^{C\text{ }(|%
\func{Re}(z\text{ })|+j\gamma _{1})}}\underset{u\text{ }\in \text{ }z\text{ }%
\mathbf{-}K_{2}^{(n)}}{\max }|h_{\text{ }0}(u)| \\ 
\leq \text{ }e^{nLe^{C\text{ }(|\func{Re}(z\text{ })|+n\text{ }\gamma _{1})}}%
\underset{u\text{ }\in \text{ }z\text{ }\mathbf{-}K_{2}^{(n)}}{\max }|\chi _{%
\text{ }-}(u)|%
\end{array}%
\right.
\end{equation*}%
Let $a>0.$\ There exists $N_{a}\in 
\mathbb{N}
^{\ast }$\ such that $(\beta _{2}+\gamma _{q-1})N_{a}\geq a$\ and :%
\begin{equation*}
\left \{ 
\begin{array}{c}
z+K_{1}^{(n)}\subset 
\mathbb{R}
_{\delta _{0}}^{+},\text{ }n\geq N_{a},\text{ }z\in 
\mathbb{R}
_{\delta _{0},a} \\ 
z-K_{2}^{(n)}\subset 
\mathbb{R}
_{\delta _{0}}^{-},\text{ }n\geq N_{a},\text{ }z\in 
\mathbb{R}
_{\delta _{0},a}%
\end{array}%
\right.
\end{equation*}%
It follows then from (\ref{split}) that we have for all $n\geq N_{a},$\ $%
z\in 
\mathbb{R}
_{\delta _{0},a}:$%
\begin{equation*}
\left \{ 
\begin{array}{c}
\begin{array}{c}
\underset{u\text{ }\in \text{ }z\text{ }\mathbf{+}K_{1}^{(n)}}{\max }|\chi _{%
\text{ }+}(u)|\leq D_{1}e^{-\cos (C_{1}\delta _{0})e^{C_{1}\underset{u\text{ 
}\in \text{ }z\text{ }\mathbf{+}K_{1}^{(n)}}{\min }|\func{Re}(u\text{ })|}}
\\ 
\leq D_{1}e^{-\cos (C_{1}\delta _{0})e^{\text{ }C_{1}(-a\text{ }+n\text{ }%
\beta _{2})}}%
\end{array}
\\ 
\underset{u\text{ }\in \text{ }z\text{ }\mathbf{-}K_{2}^{(n)}}{\max }|\chi _{%
\text{ }-}(u)|\leq D_{1}e^{-\cos (C_{1}\delta _{0})e^{C_{1}\underset{u\text{ 
}\in \text{ }z\text{ }\mathbf{-}K_{2}^{(n)}}{\min }|\func{Re}(u)|}} \\ 
\leq D_{1}e^{-\cos (C_{1}\delta _{0})e^{C_{1}(-a\text{ }+n\text{ }\gamma
_{q-1})}}%
\end{array}%
\right.
\end{equation*}%
Consequently we have for all $n\geq N_{a},$\ $z\in 
\mathbb{R}
_{\delta _{0},a}:$ 
\begin{equation*}
\left \{ 
\begin{array}{c}
|g_{n}(z)|\leq D_{1}e^{nLe^{C\text{ }(a\text{ }+n\text{ }\beta _{q})}\text{ }%
]-\cos (C_{1}\delta _{0})e^{C_{1}(-a\text{ }+n\text{ }\beta _{2})}} \\ 
|h_{n}(z)|\leq D_{1}e^{nLe^{C\text{ }(a\text{ }+n\text{ }\gamma _{1})}\text{ 
}]--\cos (C_{1}\delta _{0})e^{C_{1}(-a\text{ }+n\text{ }\gamma _{q-1})|}}%
\end{array}%
\right.
\end{equation*}%
On the other hand we have :%
\begin{equation*}
\left \{ 
\begin{array}{c}
nLe^{C\text{ }(a\text{ }+n\text{ }\beta _{q})}=\underset{n\rightarrow
+\infty }{o}\left( \cos (C_{1}\delta _{0})e^{\text{ }C_{1}(-a\text{ }+n\text{
}\beta _{2})}\right) \\ 
nLe^{C\text{ }(a\text{ }+n\text{ }\gamma _{1})}=\underset{n\rightarrow
+\infty }{o}\left( \cos (C_{1}\delta _{0})e^{\text{ }C_{1}(-a\text{ }+n\text{
}\gamma _{q})|}\right)%
\end{array}%
\right.
\end{equation*}%
Thence there exist real constants $D_{a}>0$\ and $E_{a}>0$\ and an integer $%
P_{a}\geq N_{a}$\ such that the following inequalities hold :%
\begin{equation*}
\left \{ 
\begin{array}{c}
|g_{n}(z)|\text{ }\leq D_{a}e^{-E_{a}e^{\text{ }C_{1}(-a\text{ }+n\text{ }%
\beta _{2})}},\text{ }z\in 
\mathbb{R}
_{\delta _{0},a}\text{ },\text{ }n\geq P_{a} \\ 
|h_{n}(z)|\text{ }\leq D_{a}e^{-E_{a}e^{\text{ }C_{1}(-a\text{ }+n\text{ }%
\gamma _{q})|}},\text{ }z\in 
\mathbb{R}
_{\delta _{0},a}\text{ },\text{ }n\geq P_{a}%
\end{array}%
\right.
\end{equation*}%
It follows that the function series $\sum g_{n|%
\mathbb{R}
_{\delta _{0}}}$ and $\sum h_{n|%
\mathbb{R}
_{\delta _{0}}}$ are uniformly convergent on every compact subset of $%
\mathbb{R}
_{\delta _{0}}$ and that the functions $\overset{+\infty }{\underset{n=N_{a}}%
{\sum }}g_{n}\ $and $\overset{+\infty }{\underset{n=N_{a}}{\sum }}h_{n}$ are
holomorphic on $%
\mathbb{R}
_{\delta _{0},a}$ for every $a>0.$ Let $G_{+}$ and $G_{-}$ be respectively
the sums of $\sum g_{n|%
\mathbb{R}
_{\delta _{0}}}$ and $\sum h_{n|%
\mathbb{R}
_{\delta _{0}}}.$ Since all the functions $g_{n|%
\mathbb{R}
}$ and $h_{n|%
\mathbb{R}
}$ belong to $C_{M}\{%
\mathbb{R}
\},$ it follows then that the functions $g_{+}:=G_{+}|_{%
\mathbb{R}
}$ and $g_{-}:=G_{-}|_{%
\mathbb{R}
}$ belong to $C_{M}\{%
\mathbb{R}
\}.$ Elementary computations show that : 
\begin{equation*}
\left \{ 
\begin{array}{c}
\underset{j=1}{\overset{q}{\sum }}a_{j}(x)g_{+}(x+\alpha _{j})=\chi _{+}(x),%
\text{ }x\in 
\mathbb{R}
\\ 
\underset{j=1}{\overset{q}{\sum }}a_{j}(x)g_{-}(x+\alpha _{j})=\chi _{-}(x),%
\text{ }x\in 
\mathbb{R}%
\end{array}%
\right.
\end{equation*}%
Then it follows from (\ref{SPLIT}) that the function $g:=g_{+}+g_{-}$ is a
solution on the interval $[-\delta _{0},\delta _{0}]$ of the difference
equation $(E$ $).$ But the function $x\mapsto \underset{j=1}{\overset{q}{%
\sum }}a_{j}(x)g(x+\alpha _{j})-\chi $ $(x)$ belongs to the quasianalytic
Carleman class $C_{M}\{%
\mathbb{R}
\}.$ Consequently the function $g\in C_{M}\{%
\mathbb{R}
\}$ is a solution on $%
\mathbb{R}
$ of the difference equation $(E).$ Thence the proof of the main result is
complete.

$\square $

\begin{center}
\bigskip

\bigskip \  \  \  \  \  \  \  \  \ 
\end{center}

\end{document}